\documentclass{article}
\usepackage[utf8]{inputenc}
\usepackage[english]{babel}

\usepackage{blindtext}
\usepackage{amssymb}

\usepackage{amsmath}

\usepackage{amsthm}

\usepackage[numbers]{natbib}

\newtheorem{theorem}{Theorem}

\theoremstyle{definition}

\renewcommand\qedsymbol{$\blacksquare$}

\title{The joint distribution of value and local time of simple random walk and reflected simple random walk}
\author{
Isaac Meilijson
\\
School of Mathematical Sciences,
Tel Aviv University, Israel
\\
{\em E-mail: \tt{isaco@tauex.tau.ac.il}}
\\
Yael Perlman
\\
Department of Management, Bar-Ilan University, Israel \\
{\em E-mail: \tt{yael.perlman@biu.ac.il}}
}

\date{\today}

\usepackage{graphicx}

\begin{document}

\maketitle
\renewcommand\qedsymbol{$\blacksquare$}
\begin{abstract}
The joint distribution of value and local time for Brownian Motion has been reported by Borodin and Salminen. Its asymptotic behavior for recurrent random walk has been presented by Jain and Pruitt. Motivated by the need for queue size control during a pandemic (Hassin, Meilijson and Perlman), the current study presents closed form formulas for random walk and reflected random walk with $\pm 1$ increments, not necessarily fair.

%
\end{abstract}

\section{Introduction}

Simple random walk SRW is the cumulative sum of i.i.d. $\pm 1$ increments, with some initial integer value. Reflected simple random walk RSRW is non-negative, with increments like SRW whenever positive, and increment $1$ when zero. The local time $L$ at zero up to time $n$ of either process is the number of its visits to zero up to time $n$. A formula will be derived for the joint distribution of the value of the process at time $n$ and the local time at zero up to time $n$, for SRW and RSRW. Corresponding formulas in continuous time exist for Brownian Motion (\cite{BorodinSalminen},\cite{JainPruitt}).

\bigskip

The motivation for this study is an M/M/$c$/$n$ service system observed when all servers are busy. The reflected random walk is the waiting room occupancy and local time is the number of arriving customers that fill up the queue to capacity. The joint distribution of RSRW and $L$ determine the distribution of the number of waiting customers an arriving customer will meet until entering into service. At the time of a pandemic, this distribution may provide a tool to determine a safe maximal queue size. This is a companion paper of Hassin, Meilijson and Perlman \cite{HMP} that considers Nash equilibrium, safe and socially optimal maximal queue sizes, under a utility function that take into account the negative network effects of waiting time, total time spent with waiting customers and the number of these customers, the subject matter that motivates the current study.

\bigskip

Formally, let $\epsilon_t \ ; \ t=1,2,\dots$ be i.i.d. with $P(\epsilon_t=1)=1-P(\epsilon_t=-1)=q \ , \ q \in (0,1)$, be the increments of simple random walk (SRW) $S_T=s+\sum_{t=1}^T \epsilon_t$ starting at $S_0=s$. Reflected simple random walk RSRW $R$ is defined as $R_T=s+\sum_{t=1}^T \delta_t$ (starting at $R_0=s \ge 0$), where $\delta_t=\epsilon_t$ if $R_{t-1}>0$ and $\delta_t=1$ otherwise (i.e., if $R_{t-1}=0$). The trajectories of $R$ are non-negative, increasing to $1$ whenever at zero and behaving as $S$ otherwise. $R$ can be identified with $|S|$ if $q={1 \over 2}$, but not otherwise. Let $L_0=0$ and for $t>0, L_t = L_{t-1} + I_{\{R_t=0\}}$ be the {\em local time} of RSRW $R$ at zero, the cumulative number of visits to zero in positive time.

\section{Reflected simple random walk}

Let $NT(T,s,r,l)$ be the number of $R$-trajectories of length $T \ge 1$ with $R_0=s \ge 0$, $R_T=r \ge 0$ and $L_T=l$. The ease of identifying the binomial distribution of $S_T$ (w.l.o.g $s=0$) stems from two distinct elements: All trajectories with $S_T=r$ have the same probability $q^{{T+r} \over 2} (1-q)^{{T-r} \over 2}$ and the number of such trajectories is straightforward to identify as $\binom{T}{{{T+r} \over 2}}$. This is not the case for the distribution of $R_T$, but as the next theorem shows, it is for the joint distribution of $R_T$ and $L_T$.

\begin{theorem} \label{localtimereflected}
Let $T \ge 1$ and $s, r, l$ be non-negative integers. Let $R$ be reflected simple random walk whose increments are positive with probability $q \in (0,1)$. Then

\begin{equation} \label{jointRL}
P(R_T=r, L_T=l | R_0 = s)  =  NT(T,s,r,l)q^{{{T+r-s} \over 2}-l+I_{\{r=0\}}-I_{\{s=0\}}}(1-q)^{{{T-r+s} \over 2}} \\
\end{equation}
where the number of trajectories $NT(T,s,r,l)$ is zero except for the feasible cases: $s$ any non-negative integer, $r \in \{s-T, s-T+2, s-T+4 \dots, s+T\} \cap [0,s+T], 0 \le l \le max(0,{{T-s-r} \over 2}+I_{\{s > 0\}})$. For these feasible cases the coefficients $MT(T,s,r,l)$ are as follows.

\begin{eqnarray}
NT(T,0,r,l) & = & \binom{T-l-1}{{{T+r} \over 2}-1}-\binom{T-l-1}{{{T-r} \over 2}-l-1} \nonumber \\
NT(T,s>0,r,0)& = & \binom{T}{{{T+r-s} \over 2}}-\binom{T}{{{T-r-s} \over 2}} \nonumber \\
NT(T,s>0,r,l>0) & = & \binom{T-l}{{{T-r-s} \over 2}-l+1}-\binom{T-l}{{{T-r-s} \over 2}-l} \label{jointRLtraj}
\end{eqnarray}
\end{theorem}



\begin{proof}
For initial state $s=0$ the first increment is $+1$ a.s. Hence, the probability of arriving from $s=0$ to $(r,l)$ in time $T$ is the same as the probability of arriving from $s=1$ to $(r,l)$ in time $T-1$. It is straightforward to check that the top formula in (\ref{jointRL}) derives accordingly from any of the other two, that are the same formula for $s=1$ (but not for $s>1$).

The single-trajectory probabilities are immediate to ascertain, as is the fact that all trajectories with the same triple $(s,r,l)$ are equally likely. What remains is the combinatorial part.

For $s>0, r>0, l=0$, $NT(T,s,r,0)$ is the number of non-negative trajectories of length $T$ starting at $s$ and ending at $r>0$ without ever visiting zero. Simple random walk has $\binom{T}{{{T+r-s} \over 2}}$ as many trajectories, from which those that visit zero have to be subtracted. By the usual reflection principle, the latter are just as many as the number of SRW paths from $s$ to $-r$, totalling $\binom{T}{{{T-r-s} \over 2}}$. Hence,
$NT(T,s,r,0)=\binom{T}{{{T+r-s} \over 2}}-\binom{T}{{{T-r-s} \over 2}}$, as claimed.

For $s>0,l>0$, the set to be counted (with $NT(T,s,r,l)$ elements) is the set $A(T,s,r,l)$ of non-negative trajectories of length $T$ with ${{T+r-s} \over 2}$ positive increments and $l$ visits to zero in positive time. Let $l>0$. For every element in $A(T,s,r,l)$ there is a well defined time $t$, the time of last visit to zero. Remove the increment $\delta_t=-1$ from the string of increments to obtain a shorter string $(\delta_1, \delta_2,\cdots,\delta_{t-1},\delta_{t+1},\cdots,\delta_T)$ with partial sums $s=R_0, R_1, \cdots,R_{T-1}$. This shorter RW trajectory $R$ has length $T-1$, the same number ${{T+r-s} \over 2}$ of positive increments, and last visit to $1$ at time $t-1$. Furthermore, this mapping from the set of $A(T,s,r,l)$ with last visit to zero at time $t$ to the set of $A(T-1,s,r+1,l-1)$ with last visit to $1$ at time $t-1$ is one-to-one. As long as $l-1>0$,  $t$ partitions $A(T,s,r,l)$ (differently but) just as it partitions $A(T,s,r+1,l-1)$, so this mapping is onto and these two sets have the same number of elements. Since the number of trajectories is  invariant when $T$ and $l$ are decreased by $1$ while at the same time increasing $r$ by $1$. Hence, the answer is the same as $NT(T-l+1,s,r+l-1,1)$. The plan is to carry out one more transition from $l=1$ to $l=0$. The mapping described above to $NT(T-l,s,r+l,0)$ is correct only if $s=1$ (when the existence of a visit to $1$ is guaranteed) and needs a different treatment for $s>1$, in which case the trajectories counted by $NT(T-l,s,r+l,0)$ include those that don't visit $1$ at all, whereas the one-to-one-onto mapping was only to those that visit $1$. The trajectories going from $s>1$ to $r+l$ in time $T-1$ without visiting $1$ are the same as those going from $s-1$ to $r+l-1$ in time $T-l$ without visiting $0$, i.e., $NT(T-l,s-1,r+l-1,0)$ in number.

Summarizing and performing some book-keeping,

\bigskip

$NT(T,s,r,l) = NT(T-l,s,r+l,0)=\binom{T-l}{{{T+r-1} \over 2}}-\binom{T-l}{{{T+r+1} \over 2}}$ for the case $s=1$, while for $s>1$:
\begin{eqnarray}
 & & NT(T-l,s,r+l,0) - NT(T-l,s-1,r+l-1,0) \nonumber \\
 & = & \Big[\binom{T-l}{{{T-l+r+l-s} \over 2}}-\binom{T-l}{{{T-l-r-l-s} \over 2}}\Big]-\Big[\binom{T-l}{{{T-l+r+l-1-(s-1)} \over 2}}-\binom{T-l}{{{T-l-r-l+1-(s-1)} \over 2}}\Big] \nonumber \\
 & = & \Big[\binom{T-l}{{{T+r-s} \over 2}}-\binom{T-l}{{{T-r-s} \over 2}-l}\Big]-\Big[\binom{T-l}{{{T+r-s} \over 2}}-\binom{T-l}{{{T-r-s} \over 2}-l+1}\Big] \nonumber \\
 & = & \binom{T-l}{{{T-r-s} \over 2}-l+1}-\binom{T-l}{{{T-r-s} \over 2}-l} \nonumber
\end{eqnarray}
\end{proof}

The marginal distribution of $R_T$ can be alternatively obtained by Markovian matrix methods. Let $Q$ be the transition matrix of RSRW, with all entries zero except $Q(0,1)=1$ and $Q(i,i+1)=1-Q(i,i-1)=q$ for $i \ge 1$. Now let $Q_1=Q(0:N,0:N)$ for any $N$ exceeding $T+s$. Then the marginal distribution of $R_T$ is row $s$ of $Q_1^T$. 

Expected local time at zero up to time $T$, starting at $s$, can be obtained as follows. Since it is the sum over n (from $0$ to $T$) of $Q^n(s,0)$, it is the $(s,0)$ entry of $(I-Q_1)^{-1}(I-Q_1^{T+1})$. But expectations of nonlinear functions of local time, such as its distribution, seem to be outside the standard range of algebraic methods.

\begin{figure}[hbt!]
\begin{center}
{\includegraphics[scale=0.40]{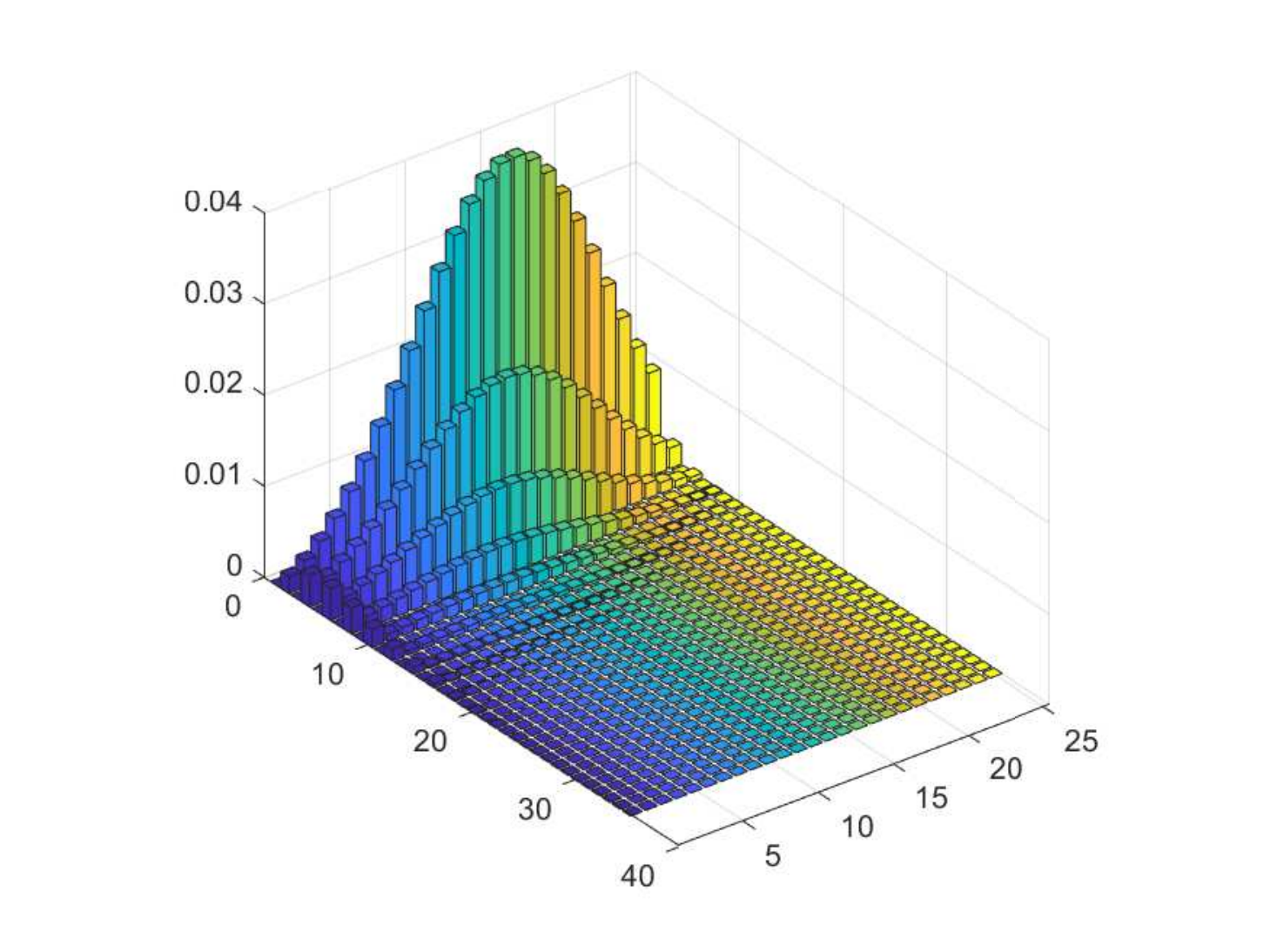}}
\qquad
{\includegraphics[scale=0.40]{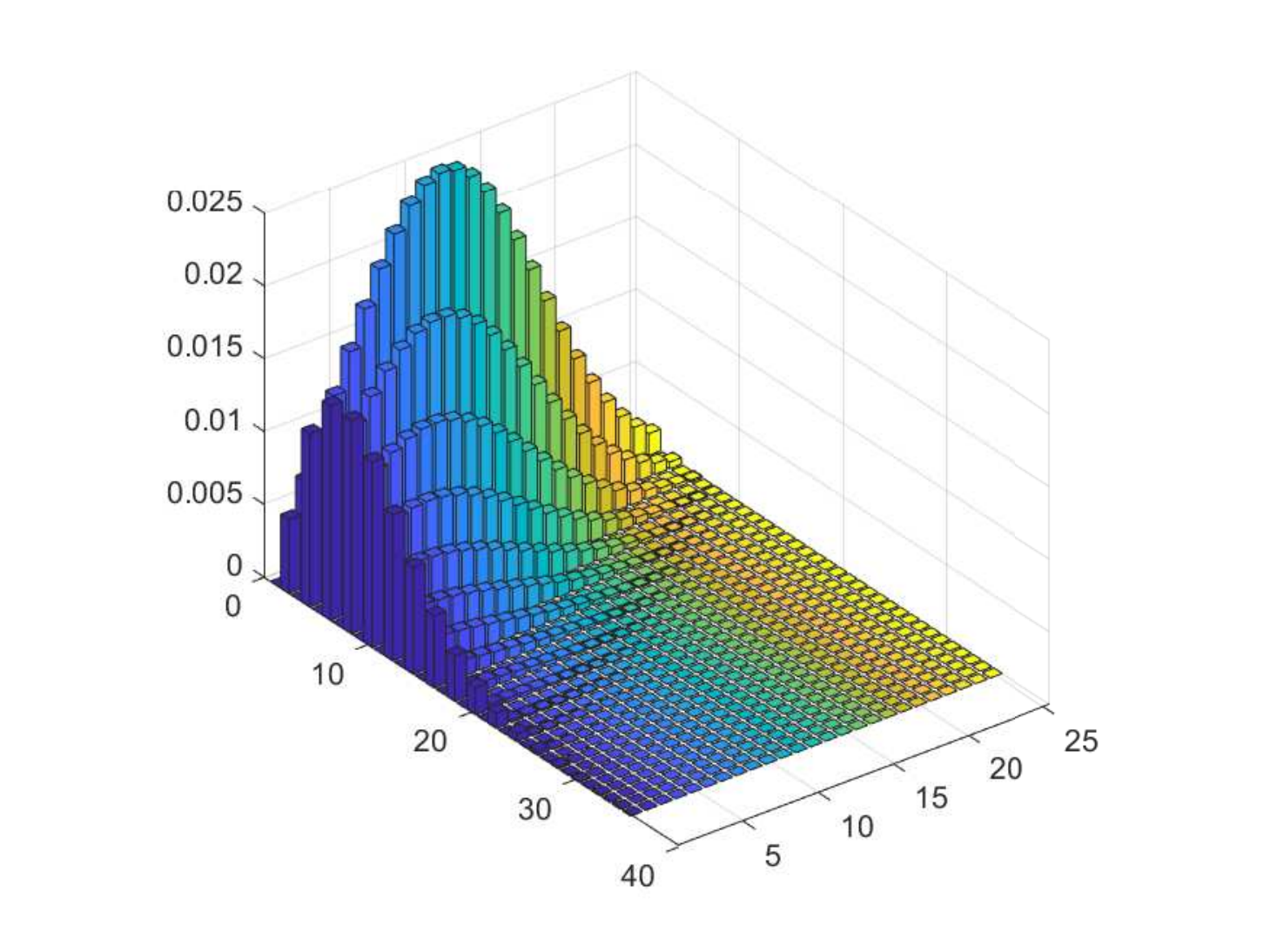}}
\qquad
{\includegraphics[scale=0.40]{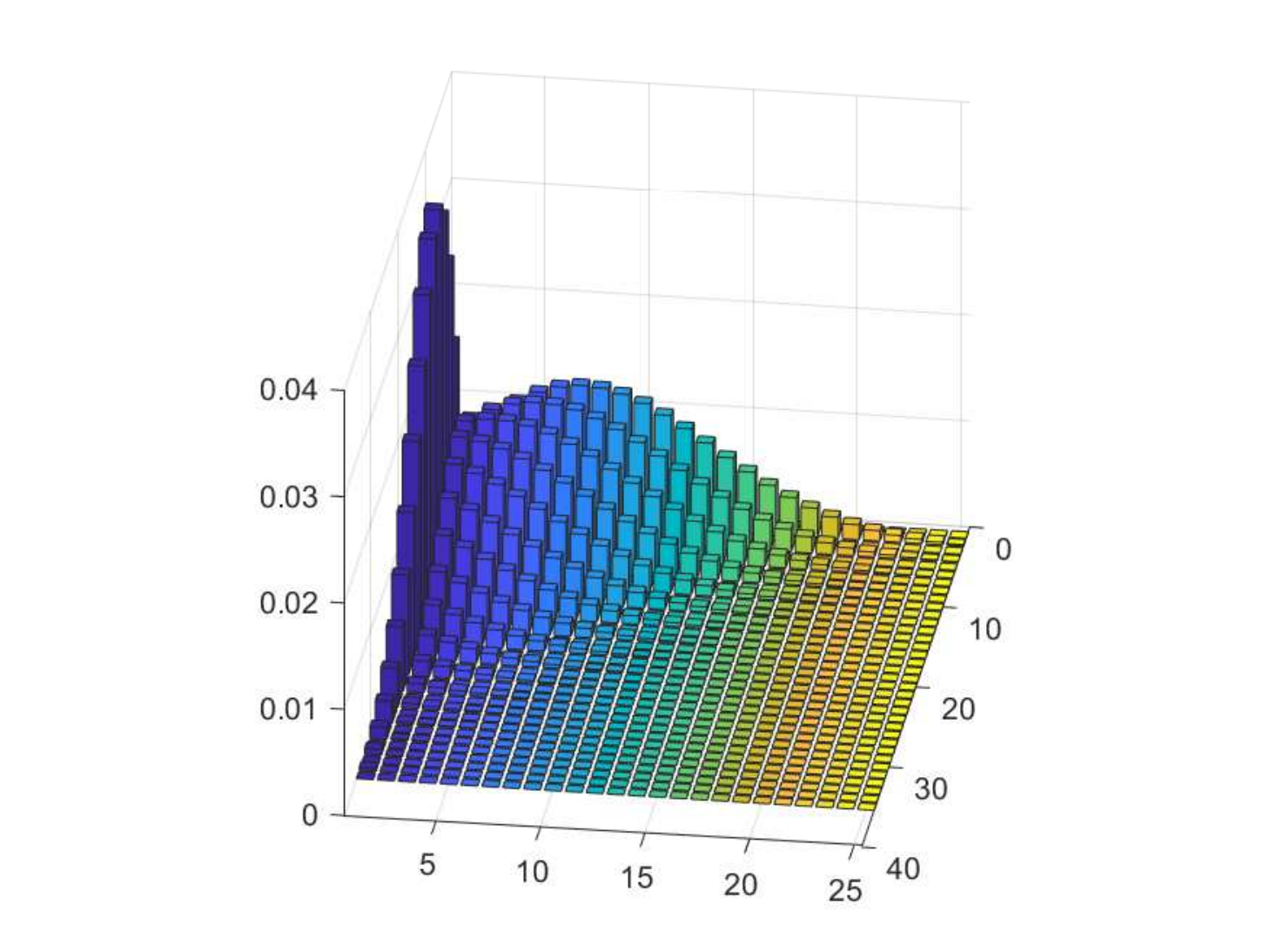}}
\qquad
{\includegraphics[scale=0.40]{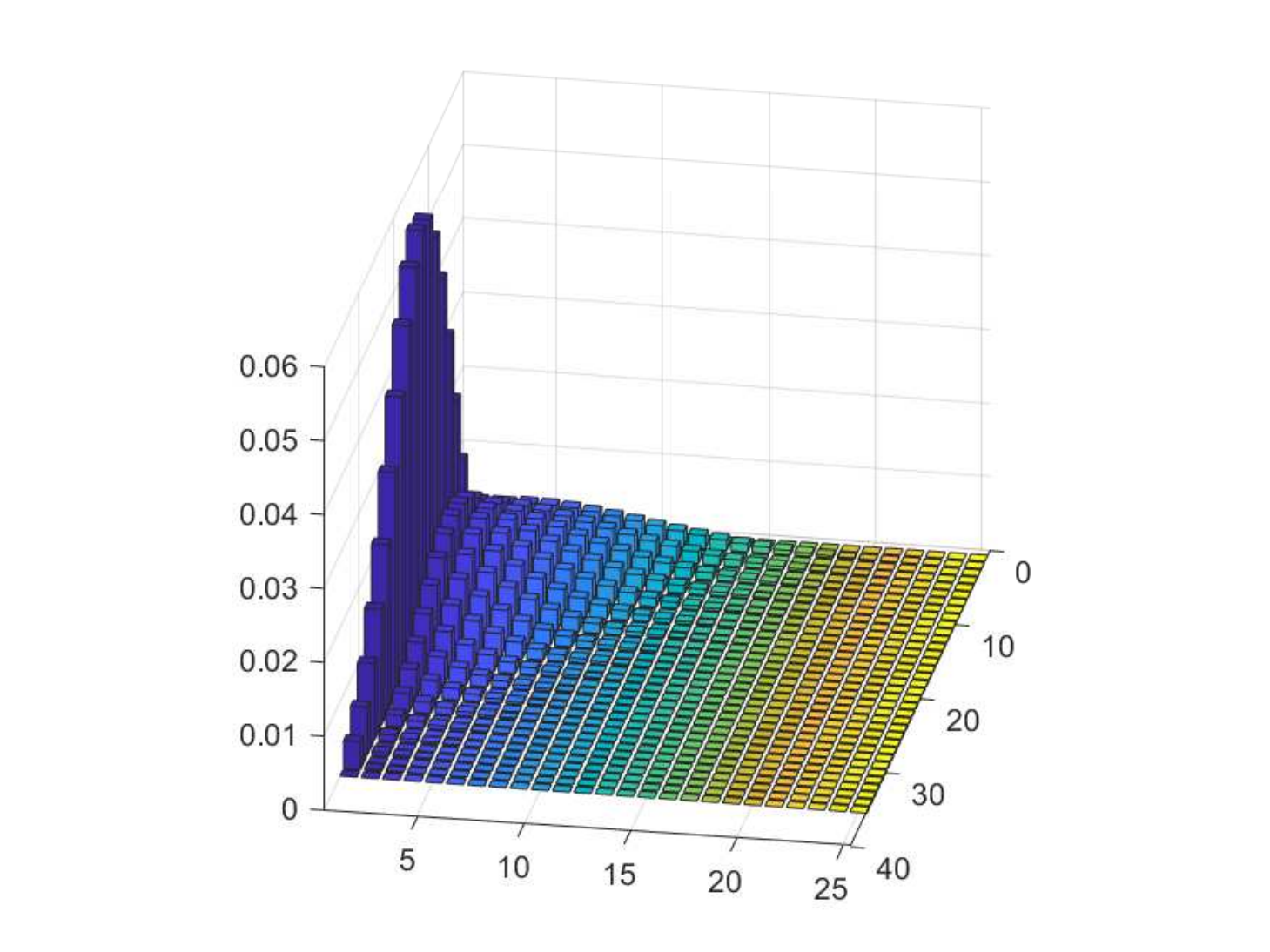}}
\end{center}
\caption{Joint distribution of value and local time at zero for reflected random walk starting at $s=3$ and ending at time $T=80$. From top left to bottom right, $q=0.40, 0.45, 0.50, 0.55$, displaying a transition from downdrift and large local time at zero to diffusion with positive drift and few returns to zero. Value has been displayed up to $35$ and local time at zero up to $25$, switching coordinates from top to bottom.}
\label{fig:jointdist}
\end{figure}

\begin{figure}[hbt!]
\begin{center}
{\includegraphics[scale=0.40]{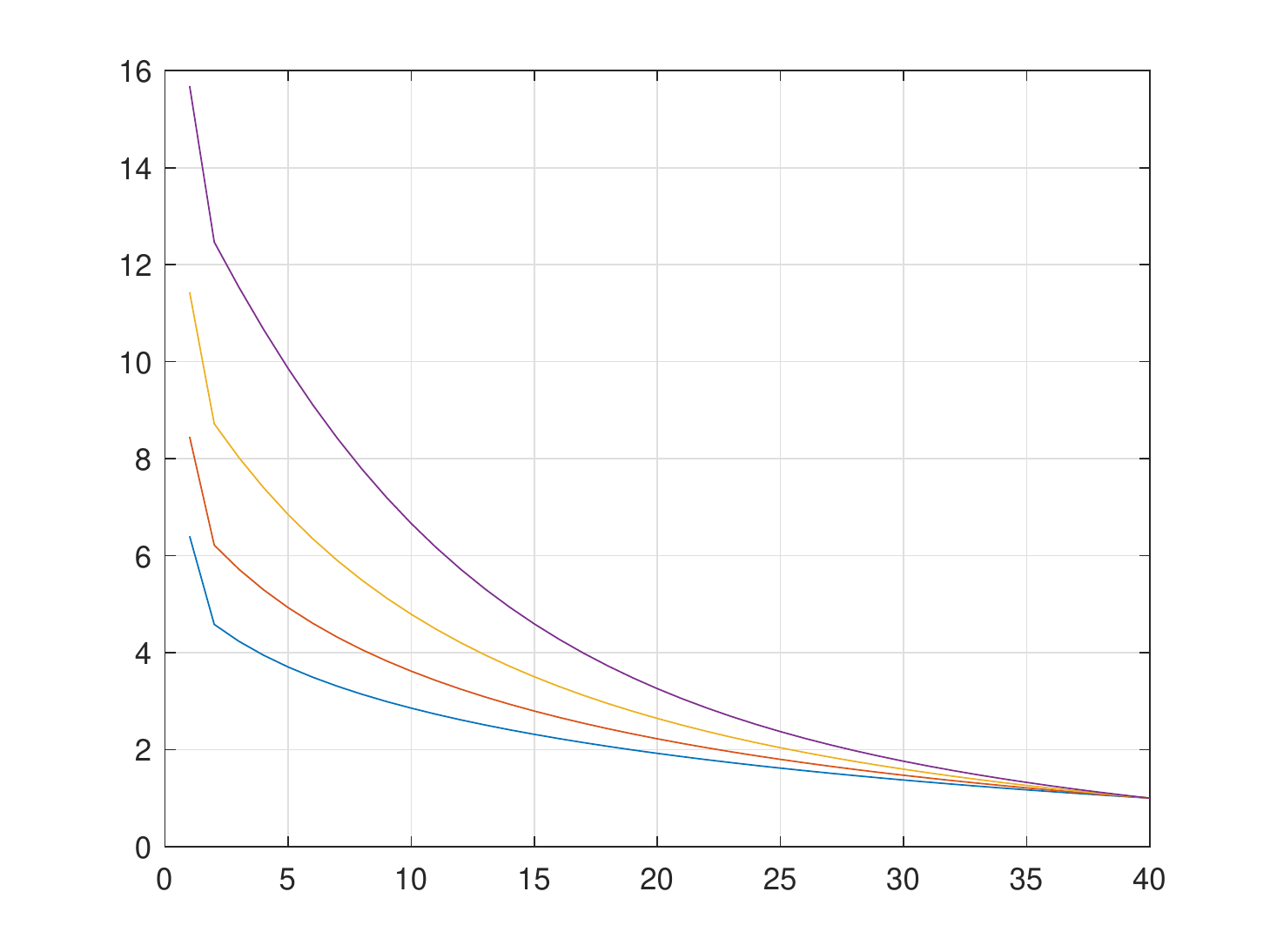}}
\qquad
{\includegraphics[scale=0.40]{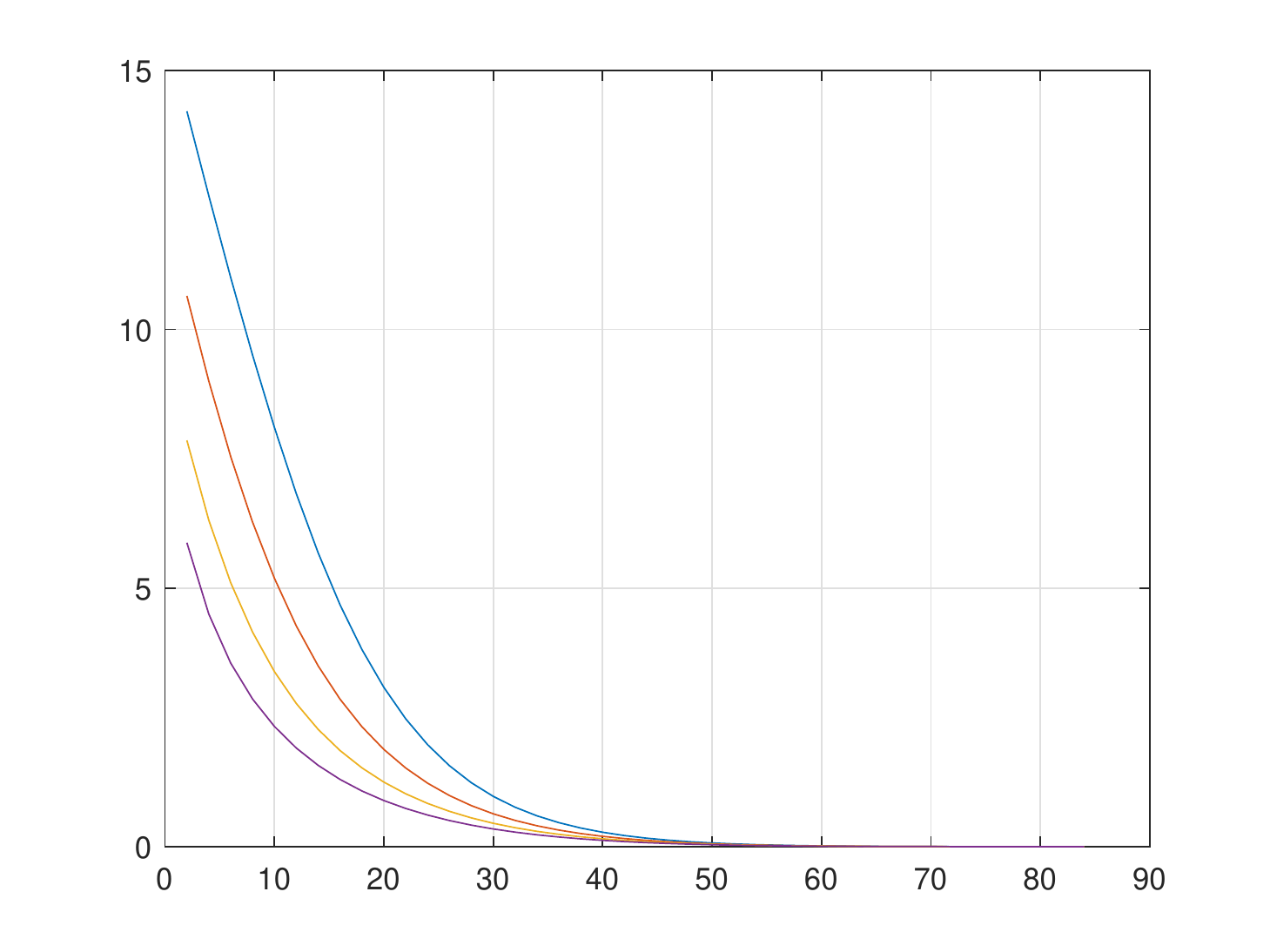}}
\end{center}
\caption{Joint distribution of value and local time at zero for reflected random walk starting at $s=3$ and ending at time $T=80$, for $q=0.40, 0.45, 0.50, 0.55$. Left, conditional expectation of local time given value ($q=0.40$ is highest). Right, conditional expectation value given local time ($q=0.40$ is lowest).}
\label{fig:condexp}
\end{figure}

\section{Simple random walk}

Rather than counting trajectories directly, their statistics will be derived from those for reflected simple random walk. In essence, each {\em solid bridge} of RSRW gives rise to two solid bridges of SRW, where a solid bridge is a sub-trajectory starting and ending at zero without zero values in between.

Adhering to the notation above, $S$ stands for SRW and $R$ stands for RSRW. Unlike RSRW where the probability of a trajectory from $s$ to $r$ depends on $l$, for SRW or simple random walk bridge $l$ can determine whether the trajectory is possible or impossible, but no internal information on the trajectory can influence the positive value of the probability because, in statistical terms, the sum of the increments is a sufficient statistic for $q$. Each trajectory from $s$ to $r$ in time $T$ has probability $q^{{{T+r-s} \over 2}}(1-q)^{{{T-r+s} \over 2}}$ and the only question is combinatorial, how to split the number $\binom{T}{{{T+r-s} \over 2}}$ of trajectories into the various subsets determined by $l$. Taking care of special cases first, the number of trajectories of length $T$ from $s \neq 0$ to $r$ with $(l=0,r \neq 0)$ or $(l=1,r = 0)$ is the same as the corresponding number of trajectories for RSRW from $|s|$ to $|r|$. The number of trajectories of length $T$ with $l=1$ from $s \neq 0$ to $r \neq 0$ is the same as the number of reflected random walk trajectories with $l=1$ from $|s|$ to $|r|$. For all other cases, $l'=l+I_{\{s=0\}}$ (augmented $l$) is at least $2$ and the trajectory has $l'-1$ solid bridges. Hence, the number of SRW trajectories is the corresponding number for RSRW (from $|s|$ to $|r|$) multiplied by $2^{l'}$.

Summarizing,

\begin{theorem} \label{localtimesimple}
Let $T \ge 2$ and $s, r, l$ be integers (with $l \ge 0$). Let $S$ be simple random walk whose increments are positive with probability $q \in (0,1)$. Then
\begin{equation} \label{jointSL}
P(S_T=r, L_T=l | S_0 = s) = MT(T,s,r,l)q^{{{T+r-s} \over 2}}(1-q)^{{{T-r+s} \over 2}}
\end{equation}
where the number of trajectories $MT(T,s,r,l)$ is zero except for the feasible cases: $s$ any integer, $r \in \{s-T, s-T+2, s-T+4, \dots, s+T\}, 0 \le l \le max(0,{{T-|s|-|r|} \over 2}+I_{\{s \neq 0\}})$. For these feasible cases the coefficients $MT(T,s,r,l)$ are expressed in terms of $l'=l+I_{\{s=0\}}$ as

\begin{equation}
\label{jointRLtrajsimple}
MT(T,s,r,l) = \begin{cases}
\begin{aligned}
    2^l \bigg [\binom{T-l-1}{{{T+|r|} \over 2}-1}-\binom{T-l-1}{{{T+|r|} \over 2}}\bigg] & ~s=0 \\
    \binom{T}{{{T+|r|-|s|} \over 2}}-\binom{T}{{{T-|r|-|s|} \over 2}} & ~s \neq 0, l=0 \\
    2^{l-1} \bigg [\binom{T-l}{{{T+|r|+|s|} \over 2}-1}-\binom{T-l}{{{T+|r|+|s|} \over 2}}\bigg] & ~s \neq 0, l>0
\end{aligned}
\end{cases}
\end{equation}

\end{theorem}

\bigskip

\section{Mixture distributions} 

As expressed by (\ref{jointRLtraj}) and (\ref{jointRLtrajsimple}), and illustrated in Figures \ref{fig:jointdist} and \ref{fig:bothcases}, each of the two joint distributions (\ref{jointRL}) and (\ref{jointSL}) of value at time $T$ and local time at zero up to time $T$, starting at $s>0$, is a mixture of two (discrete but) smooth distributions, one modelling paths with no visits to zero ($L=0$, middle term in (\ref{jointRLtraj})) and the other covering all paths that visit zero ($L>0$, last term in (\ref{jointRLtraj})). Figure \ref{fig:bothcases} displays RSRW (left) and SRW (right) for $T=160, s=5,q=0.45$. 

\bigskip

{\bf The univariate mixture component $\{L=0\}$ is log-concave in $r$}. This distribution is analyzed in detail for even $r$. The case of odd $r$ only requires a minor adjustment in the auxiliary terms $W$ and $x$ to be defined. Consider the logarithm of the probability (\ref{jointRL}) as a function of $r$, for fixed $T, q, l=0$ and $s>0$. It is the sum of $\log(NT)$ and a linear function of $r$. Hence, if $\log(NT)$ is a concave function of $r$, the $\{L=0\}$-component of (\ref{jointRL}) is log-concave for each $q$. To show concavity of $\log(NT)$ means to show that the increment $\log(NT(T,s,r+2,0))-\log(NT(T,s,r,0))$ is a decreasing function of $r$. In other words, that ${{NT(T,s,r+2,0)} \over {NT(T,s,r,0)}} \le {{NT(T,s,r,0)} \over {NT(T,s,r-2,0)}}$. The proof will be finished by expressing the cross-product $NT(T,s,r,0)^2-NT(T,s,r-2,0)NT(T,s,r+2,0)$ as the sum of three non-negative terms, where $W={{T-s} \over 2}<{{T} \over 2}$ and $x={r \over 2}$. These are formal equalities. Whenever the RHS is not well defined (non-positive terms in the denominator), the LHS is non-negative, as needed.

\begin{eqnarray}
\binom{T}{W+x}^2-\binom{T}{W+x-1} \binom{T}{W+x+1}& = &\binom{T}{W+x}^2 {{T+1} \over {(W+x+1)(T-W-x+1)}} \nonumber \\
\binom{T}{W-x}^2-\binom{T}{W-x-1} \binom{T}{W-x+1}& = &\binom{T}{W-x}^2 {{T+1} \over {(W-x+1)(T-W+x+1)}} \nonumber 
\end{eqnarray}
\begin{eqnarray}
 & & \binom{T}{W-x+1} \binom{T}{W+x+1}+\binom{T}{W+x-1} \binom{T}{W-x-1}-2\binom{T}{W+x} \binom{T}{W-x} \nonumber \\
 & = & (T+1)(T-1-2 W)[{1 \over {(W+1)^2-x^2}}+{1 \over {(T-W+1)^2-x^2}}]  \label{sausage}
\end{eqnarray}

Log concavity of the bivariate mixture component $\{L>0\}$ in $(r,l)$ is under study, and the report will be updated accordingly. Concavity on a finite subset of $R_2$ requires special treatment.Kannai \cite{kannai} is a good start.


\begin{figure}[hbt!]
\begin{center}
{\includegraphics[scale=0.40]{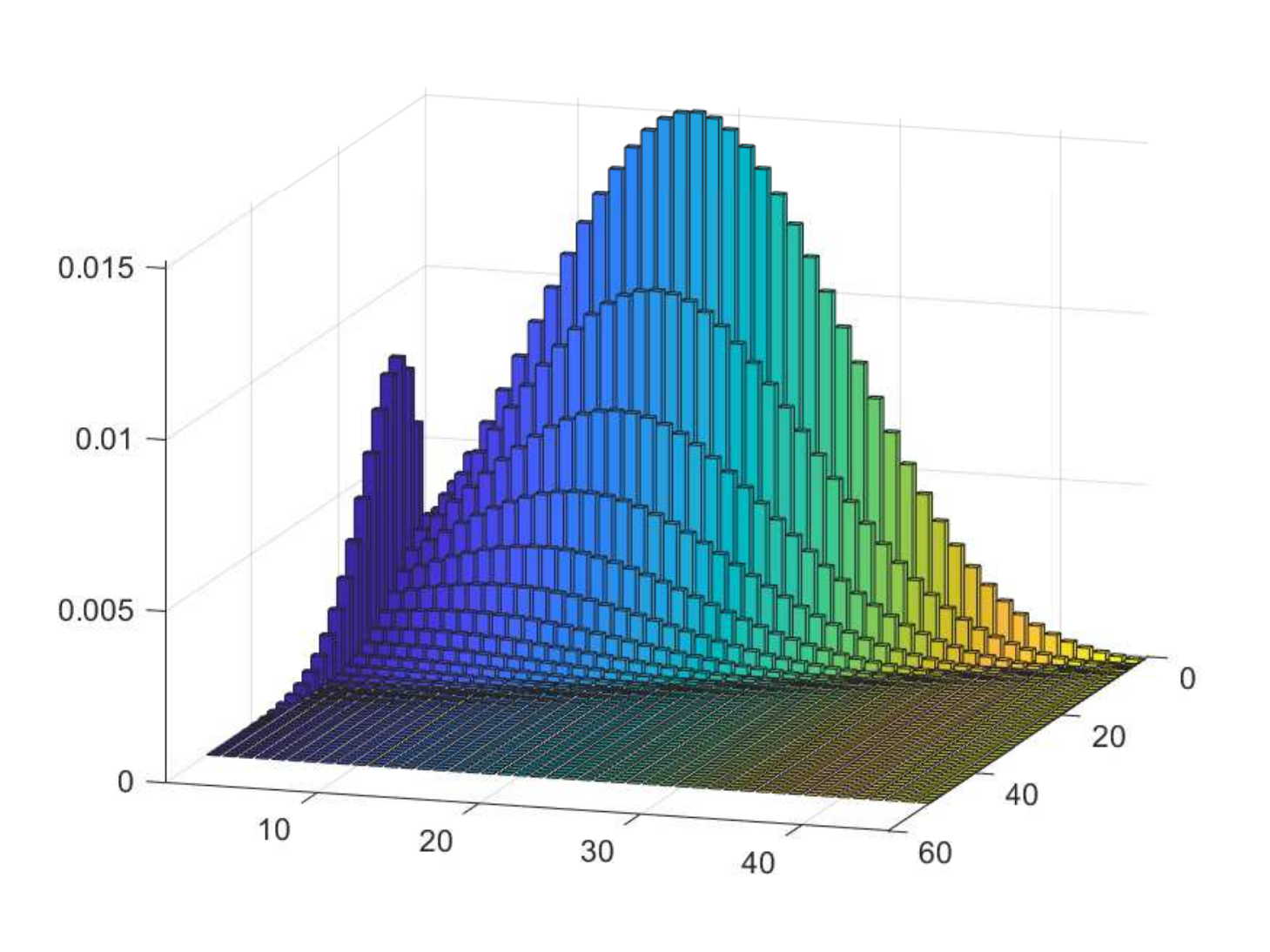}}
\qquad
{\includegraphics[scale=0.40]{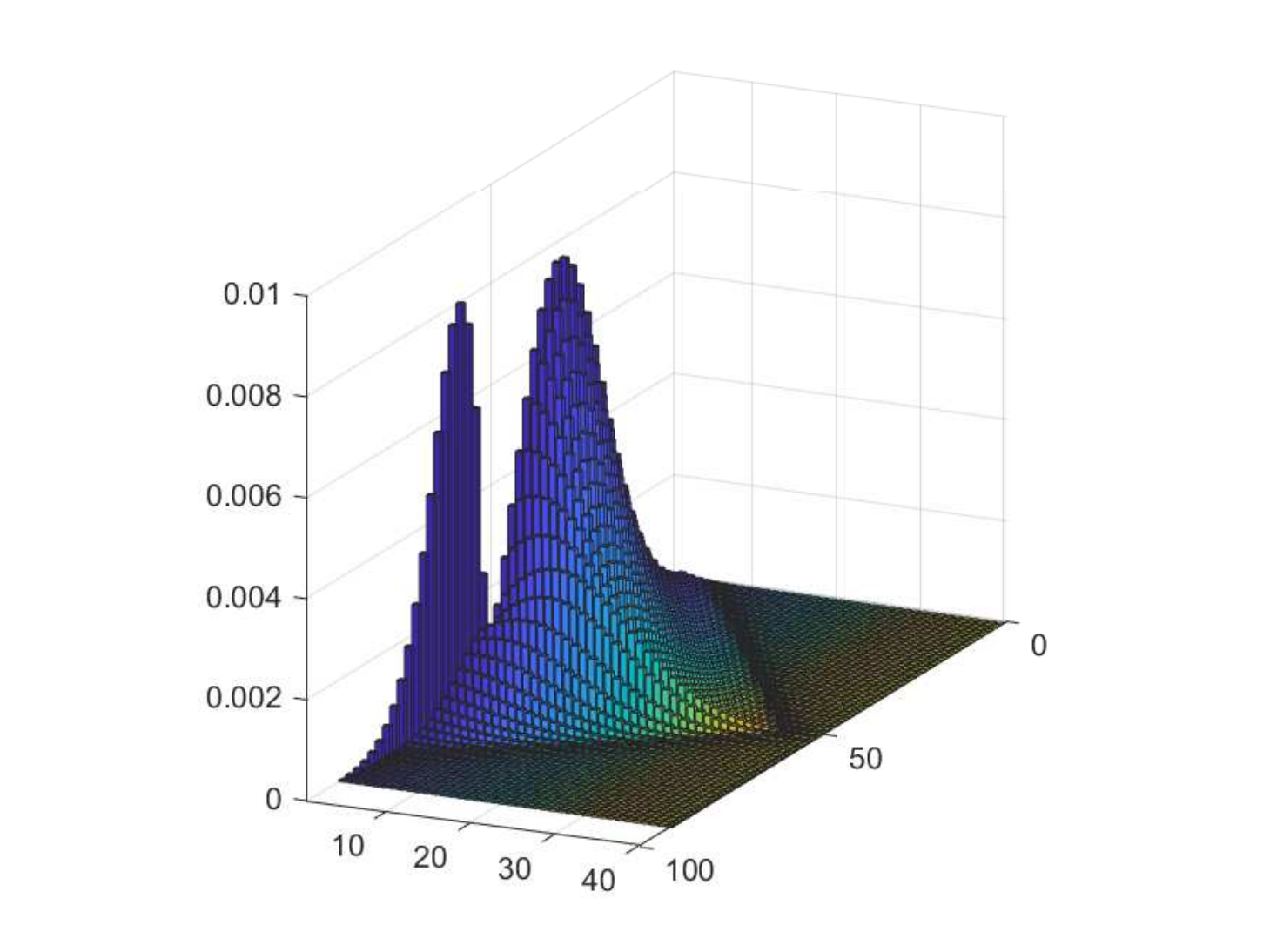}}
\end{center}
\caption{Joint distribution of value and local time at zero for reflected (left) and simple (right) random walk starting at $s=5$ and ending at time $T=160$, for $q=0.45$. Both graphs display mixtures of two distributions, one with zero local time and the other with visits to zero. The value scale of SRW has zero at coordinate 60.}
\label{fig:bothcases}
\end{figure}

\section*{Funding} This research was supported by the Israel Science Foundation, grant No. 1898/21.

\bibliographystyle{plain}


\end{document}